\documentclass[a4paper,12pt]{amsart}
\usepackage[utf8]{inputenc}
\usepackage{hyperref}
\usepackage{fullpage}
\usepackage{amsmath, amssymb}
\usepackage{mathtools} 
\usepackage[alphabetic,nobysame, initials]{amsrefs}
\usepackage{graphicx}
\usepackage{comment}
\usepackage{color}
\usepackage{leftidx}

\newtheorem{thm}{Theorem}[section]
\newtheorem{lem}[thm]{Lemma}
\newtheorem{cor}[thm]{Corollary}
\theoremstyle{definition}

\renewcommand{\Re}{\mathbb R}
\newcommand{\Ze}{\mathbb Z}
\newcommand{\Zep}{\Ze^{\geq 0}}
\renewcommand{\epsilon}{\varepsilon}
\newcommand{\Red}{\Re^d}

\newcommand{\FF}{\mathcal F}

\newcommand{\st}{\; : \; }
\renewcommand{\phi}{\varphi}
\newcommand{\vol}[1]{\operatorname{vol}\left(#1\right)}

\newcommand{\card}[1]{\left|#1\right|}

\newcommand{\texp}[2][\lambda]{
{^{#2}{#1}}
}

\newcommand{\remark}[1]{}

\newcommand{\myconst}{3.153}

\title{Approximating set multi-covers}
\author[M. Nasz\'odi, A. Polyanskii]
{M\'arton Nasz\'odi \and Alexandr 
Polyanskii}
\address[M. Nasz\'odi]{
ELTE, Dept. of Geometry,
Lorand E\"otv\"os University,
P\'azm\'any P\'eter S\'et\'any 1/C
Budapest, Hungary 1117
}

\address[A. Polyanskii]{
Moscow Institute of Physics and Technology, Technion, Institute for 
Information Transmission Problems RAS.
}

\email{marton.naszodi@math.elte.hu \and alexander.polyanskii@yandex.ru}

\thanks{M. Nasz\'odi thanks the following agencies for their support: 
the Swiss National Science Foundation grants no. 200020-162884 and 
200021-175977;
the J\'anos Bolyai Research Scholarship of the Hungarian Academy of Sciences; 
and the National Research, Development and Innovation Office, NKFIH Grants 
PD-104744 and K119670.
}
\thanks{A. Polyanskii was partially supported by the Russian Foundation for 
Basic 
Research, grants 15-31-20403 (mol\_a\_ved), 
15-01-99563 A, 15-01-03530 A}

\begin{document}
\begin{abstract}
Johnson and Lov\'asz and Stein proved independently that any hypergraph 
satisfies $\tau\leq (1+\ln \Delta)\tau^{\ast}$, where $\tau$ is the transversal 
number, $\tau^{\ast}$ is its fractional version, and $\Delta$ denotes the 
maximum degree. We prove $\tau_f\leq \myconst\tau^{\ast}\max\{\ln \Delta, f\}$ 
for 
the $f$-fold transversal number $\tau_f$.
Similarly to Johnson, Lov\'asz and Stein, we also show that this bound can be 
achieved non-probabilistically, using a greedy algorithm. 

As a combinatorial application, we prove an estimate on how fast $\tau_f/f$ 
converges to $\tau^{\ast}$. As a geometric application, we obtain an upper 
bound on the minimal density of an $f$-fold covering of the $d$-dimensional 
Euclidean 
space by translates of any convex body. 
\end{abstract}

\maketitle

\section{Introduction and Preliminaries}

A \emph{hypergraph} is a pair $(X,\FF)$, where $X$ is a finite set and 
$\FF\subseteq 2^X$ is a family of some subsets of $X$. 
We call the elements of $X$ \emph{vertices}, and the members of $\FF$ 
\emph{edges} of the hypergraph.
When a vertex is contained in an edge, we may say that 'the vertex covers the 
edge', or that 'the edge covers the vertex'.

Let $f$ be a positive integer.
An \emph{$f$-fold transversal} of $(X,\FF)$ is a multiset $A$ of $X$ such that 
each member of $\FF$ contains at least $f$ elements (with multiplicity).
The \emph{$f$-fold transversal number} $\tau_f$ of $(X,\FF)$ is the minimum 
cardinality (with multiplicity) of an $f$-fold transversal. 
A 1-transversal is called a transversal, and the 1-transversal number is called 
the transversal number, and is denoted by $\tau=\tau_1$. 

A \emph{fractional transversal} is a function $w:X\longrightarrow 
[0,1]$ with $\sum_{x:x\in F} w(x)\geq 1$ for all $F\in\FF$. The 
\emph{fractional transversal number} of $(X,\FF)$ is 
\begin{equation*}
\tau^\ast= \tau^{\ast}(\FF):=\inf\left\{ \sum_{x:x\in X} w(x) \st w \mbox{ is a 
fractional transversal} \right\}.
\end{equation*}

Clearly, $\tau^{\ast}\leq\tau$. In the opposite direction, 
Johnson \cite{Jo74}, Lov\'asz \cite{Lo75} and Stein \cite{St74} independently 
proved that 
\begin{equation}\label{eq:lovasz}
\tau\leq
(1+\ln \Delta)\tau^{\ast}, 
\end{equation}
where $\Delta$ denotes the \emph{maximum degree} of $(X,\FF)$, that is, the 
maximum number of edges a vertex is contained in. They showed that the 
greedy algorithm, that is, picking vertices of $X$ one by one, in such a way 
that we always pick one that is contained in the largest number of uncovered 
edges, yields a transversal set whose cardinality does not exceed the right 
hand side in \eqref{eq:lovasz}. For more background, see Füredi's survey 
\cite{Fu88}.

Our main result is an extension of this theorem to $f$-fold transversals.

\begin{thm}\label{thm:ftransversal}
Let $\lambda\in(0,1)$ and let $f$ be a positive integer. Then, with 
the above notation, 
\begin{equation}\label{eq:ftransversalthm}
\tau_f\leq \frac{1-\lambda^f}{1-\lambda}\tau^{\ast} (1+\ln \Delta -(f-1)\ln 
\lambda),
\end{equation}
moreover, for rational $\lambda$, the greedy algorithm using appropriate 
weights, yields an $f$-fold transversal of cardinality not exceeding the right 
hand side of \eqref{eq:ftransversalthm}.
\end{thm}
Substituting $\lambda=0.287643$ (which is a bit less than $1/e$), we obtain
\begin{cor}
With the above notation, we have
\begin{equation}\label{eq:ftransversalcor}
\tau_f\leq
 \myconst\tau^{\ast}\max\{\ln \Delta, f\}.
\end{equation}
\end{cor}

This result may be interpreted in two ways. First, it gives an algorithm that 
approximates the integer programming (IP) problem of finding $\tau_f$, with a 
better 
bound on the output of the algorithm than the obvious estimate $\tau_f\leq 
f\tau\leq f\tau^* (1+\ln \Delta)$.

A similar result was obtained by Rajagopalan and Vazirani in
\cite{RV98} (an improvement of \cite{Do82}), where, the (multi)-set 
(multi)-cover problem is considered, that is, the goal is to cover vertices by 
sets. This is simply the combinatorial dual (and therefore, equivalent) 
formulation of our problem. In 
\cite{RV98}, each set can be chosen at most once. They present generalizations 
of the greedy algorithm of \cite{Jo74}, \cite{Lo75} and \cite{St74}, and prove 
that it finds an approximation of the (multi)-set (multi)-cover problem within 
an $\ln \Delta$ factor of the optimal solution of the corresponding linear 
programming (LP) problem. Moreover, they give parallelized versions of the 
algorithms.

The main difference between \cite{RV98} and the present paper is that there, 
the optimal solution of an IP problem is compared to the optimal solution of 
the LP-relaxation of the same IP problem, whereas here, we compare $\tau_f$ 
with $\tau^*$, where the latter is the optimal solution of a weaker LP problem: 
the problem with $f=1$.

We note that, using the fact that $f\tau^{\ast}\leq\tau_f$, 
\eqref{eq:ftransversalcor} also implies that the \emph{performance ratio} (that 
is, the ratio of the value obtained by the algorithm to the optimal value, in 
the worst case) of our algorithm is constant when $\ln\Delta\leq f$. Compare 
this with \cite{BDS04}*{Lemma~1 in Section~3.1}, where it is shown that, even 
for large $f$, the standard greedy algorithm yields a performance ratio of 
$\Omega(\ln m)$, where $m$ is the number of sets in the hypergraph. Further 
recent results on the performance ratio of another modified greedy algorithm 
for variants of the set cover problem can be found in \cite{FK06}. See also 
Chapter~2 of the book \cite{Va01} by Vazirani. 

The second interpretation of our result is the following. It is easy to see 
that $\frac{\tau_f}{f}$ converges to $\tau^{\ast}$ as $f$ tends to infinity. 
Now, \eqref{eq:ftransversalcor} quantifies the speed of this convergence in 
some sense. In particular, it yields that for $f=\ln \Delta$ we have 
$\frac{\tau_f}{f}\leq \myconst\tau^{\ast}$. We have better approximation for 
larger $f$.

\begin{cor}\label{cor:onexistance_of_good_fractional_transvesal}
For every $0<\varepsilon\leq 1$, if we set 
$f:=\left\lceil\frac{2(1+\ln\Delta)}{\varepsilon(1-\lambda)}\right\rceil$, where
$0<\lambda<1$ is such that 
$-\ln \lambda/(1-\lambda)\leq 1+\varepsilon /2$, then the $f$-fold transversal 
constructed in Theorem~\ref{thm:ftransversal} yields a fractional transversal 
which gives
\[
\tau^{\ast}\leq \frac{\tau_f}{f}
\leq \tau^{\ast}(1+\varepsilon).
\]
\end{cor}

We prove Theorem~\ref{thm:ftransversal} and 
Corollary~\ref{cor:onexistance_of_good_fractional_transvesal} in 
Section~\ref{sec:proofmain}, where, at the end, we discuss the running time of 
our algorithm.

\subsection{A geometric application}
Next, we turn to a classical geometric covering problem.
Rogers \cite{Ro57} showed that for any convex body $K$ in $\Red$, there is a 
covering of $\Red$ with translates of $K$ of density at most
\begin{equation}\label{eq:rogers}
d\ln{d}+d\ln\ln d+5d.
\end{equation}
For the definition of density cf. \cite{PaAg95}.
G. Fejes T\'oth \cites{FTG76,FTG79} gave the non-trivial lower bound 
$c_d f$ for the density of an $f$-fold covering of $\Red$ by Euclidean unit 
balls (with some $c_d>1$). For more information on multiple coverings in 
geometry, see the survey \cite{FTGhandbook}.
As an application of Theorem~\ref{thm:ftransversal}, we give a similar estimate 
for $f$-fold coverings.

\begin{thm}\label{thm:multiplecovspace}
Let $K\subseteq \Red$ be a convex body and $f\geq 1$ an integer. Then there is 
an arrangement of translates of $K$ with density at most 
\[
(1+o(1)) \cdot 3.153\max\left\{
d\ln d,
f
\right\},
\]
where every point of $\Red$ is covered at least $f$ times.
\end{thm}

The key in proving Theorem~\ref{thm:multiplecovspace} is a general statement, 
Theorem~\ref{thm:multiplecovgeneral}, presented in 
Section~\ref{sec:geometry}. Both theorems are proved the same way as 
corresponding 
results in \cite{N15}, where the case $f=1$ is considered.

Earlier versions of Theorems~\ref{thm:multiplecovspace} and 
\ref{thm:multiplecovgeneral} were proved in \cite{FNN}. There, in place of 
the main result of the present paper, a probabilistic argument is used which 
yields quantitatively weaker bounds. The quantitative gain here comes from the 
fact that in the probabilistic bound on $\tau_f$ presented in \cite{FNN}, one 
has the size of the edge set $\FF$ as opposed to the maximum degree $\Delta$, 
which is what we have in \eqref{eq:ftransversalcor}.

\section{Proof of Theorem~\ref{thm:ftransversal}}\label{sec:proofmain}

\subsection{The algorithm.}
First, we imagine that each member of $\FF$ has $f$ bank notes, the 
denominations are $\$1,\$\lambda,\ldots,\$\lambda^{f-1}$, where $\lambda<1$ is 
fixed. We pick vertices one by one. At each step, we pick a vertex, 
and each edge that contains it pays the largest bank note that it has. So, each 
edge pays $\$1$ for the first vertex selected from it, then $\$\lambda$ for the 
second, etc., and finally, $\$\lambda^{f-1}$ for the $f$-th vertex that it 
contains. Later on, it does not pay for any additional selected vertex that it 
contains. Now, we follow the greedy algorithm: at each step, we pick the vertex 
that yields the largest payout at that step. We finish once each edge is 
covered at least $f$ times, that is, when we collected all the money.

\subsection{Notation}
Given a positive integer $f$, we define the \emph{truncated exponential 
function} denoted by $\texp{k}$ as follows: for any $\lambda>0$, and any $0\leq 
k<f$, let $\texp{k}=\lambda^k$, and let $\texp{k}=0$ for any $k\geq f$.
Note that the value of $f$ is implicitly present in any formula involving the 
truncated exponential function.

For each $F\in\FF$, let $k(F)$ denote the number of chosen vertices 
(with multiplicity) contained in $F$. We call the function $k:\FF\to\Zep$ the 
\emph{current state}, where $\Zep$ is the set of non-negative integers. At the 
start, $k$ is identically zero.

Given a function $k:\FF\to\Zep$,
we define the \emph{value of a vertex} $x\in X$ with respect to $k$ as
\begin{equation*}
 v_k(x):=\sum_{F:x\in F\in 
\FF} \texp{k(F)}.
\end{equation*}
The \emph{total remaining value} of $k$ is defined as
\begin{equation*}
 v(k):=\sum_{F:F\in \FF}{\mathop\sum\limits_{i=k(F)}^{f} 
\texp{i}},
\end{equation*}
which is the total pay out that will be earned in the subsequent steps.

\subsection{Fractional matchings}
A \emph{fractional matching} of the hypergraph $(X,\FF)$  is a function 
$w:\FF\longrightarrow [0,1]$ with $\sum_{F:x\in F\in\FF} w(F)\leq 1$ for all 
$x\in X$. The \emph{fractional matching number} of $(X,\FF)$ is 
\begin{equation*}
\nu^\ast= \nu^{\ast}(\FF):=\sup\left\{ \sum_{F:F\in \FF} w(F) \st w \mbox{ is 
a 
fractional matching} \right\}.
\end{equation*}
By the duality of linear programming, $\nu^{\ast}=\tau^{\ast}$.

We will need the following simple observation.

\begin{lem}\label{lem:barnu}
Let $z>0$, and $\ell:\mathcal F\to \Zep$ be such that $v_{\ell}(x)\leq z$ 
for any $x\in X$. Then we have
\[
   \frac{v(\ell)}{z}\leq
(1+\lambda+\ldots+\lambda^{f-1})\nu^{\ast}(\FF)=\frac{1-\lambda^f}{1-\lambda}
\nu^{\ast}(\FF).
    \]
\end{lem}

\begin{proof}[Proof of Lemma~\ref{lem:barnu}]
Let
\begin{equation*}
 w(F):=\frac{\mathop\sum\limits_{i=\ell(F)}^{f} 
\texp{i}}{z(1+\lambda+\ldots+\lambda^{f-1})},\text{ for any } 
F\in\FF.
\end{equation*}
First, we show that $w$ is a fractional matching. Indeed, fix an $x\in X$.
\begin{equation*}
 \sum_{F:x\in F\in\FF}w(F)=
 \frac{1}{z} \sum_{F:x\in F\in\FF}
 \frac{\mathop\sum\limits_{i=\ell(F)}^{f} 
\texp{i}}{1+\lambda+\ldots+\lambda^{f-1}}
 \leq
\end{equation*}
\begin{equation*}
 \leq \frac{1}{z} \sum_{F:x\in F\in\FF} 
\texp{\ell(F)}=
 \frac{v_{\ell}(x)}{z}\leq 1.
\end{equation*}
Second, the total weight is
\begin{equation*}
\sum_{F:F\in\FF}w(F)=
\frac{1}{z(1+\lambda+\ldots+\lambda^{f-1})}\sum_{F:F\in\FF} 
{\mathop\sum\limits_{i=\ell(F)}^{f} 
\texp{i}}=
\frac{v(\ell)}{z(1+\lambda+\ldots+\lambda^{f-1})},
\end{equation*}
finishing the proof of the Lemma.
\end{proof}

\subsection{Finally, we count the steps of the algorithm.}
We may assume that $\lambda=p/q\in (0,1)$ with $p,q\in\Ze^+$. 
If $\lambda$ is irrational, then the statement of 
Theorem~\ref{thm:ftransversal} follows by continuity.
Clearly, $q^{f-1}$ is a common denominator for the pay outs at each step.

At the start, $k_0(F):=k(F)=0$ for all $F\in\FF$.
We group the steps according to the $\$$-amount (that is, $v_{k}(x)$) 
that we get at each.

In the first $t_1$ steps, each vertex $x$ that we pick has value 
$v_{k}(x)=\Delta=:z_1$, where, we recall, $\Delta$ is the maximum degree in the 
hypergraph. Let $k_1:\FF\to\Zep$ denote the current state after the 
first $t_1$ steps.

Then, in the second group of steps, we make $t_2$ steps, at each picking a 
vertex $x\in V$ of value $v_{k}(x)=\Delta-q^{1-f}=:z_2$, where $k$ changes at 
each step. Let $k_2:\FF\to\Zep$ denote the current state after the first 
$t_1+t_2$ steps.

In the $j$-th group of steps, we make $t_j$ 
steps, at each picking a vertex $x\in V$ of value 
$v_{k}(x)=\Delta-(j-1)q^{1-f}=:z_j$.
Let $k_j:\FF\to\Zep$ denote the current state after the 
first $t_1+\ldots+t_j$ steps.

Obviously, $t_j\geq 0$, moreover some $t_j$ may be zero. For instance (the 
reader may check as an exercise), if $f>1$, then $t_2=0$. For the last group, 
we have $j=q^{f-1}\Delta-p^{f-1}+1=:N$. 

Notice that $v_{k_j}(x)\leq z_{j+1}$ for any $x\in V$. 
Therefore, by Lemma~\ref{lem:barnu}, we have
\begin{equation} \label{eq:vj<=v*}
\frac{v(k_j)}{z_{j+1}}\leq
\frac{1-\lambda^f}{1-\lambda}\nu^{\ast}(\mathcal{F}).
\end{equation}
Clearly,
\begin{equation}\label{eq:vj=sumtj}
v(k_j)=\sum_{i=j+1}^{N} t_i z_i,\;\;\text{ for } 0\leq j 
\leq N-1.
\end{equation}
In total, we choose $t_1+t_2+\dots+t_N$ vertices (that is 
the cardinality of $A$ with multiplicity), and they form an $f$-fold 
transversal of $(X,\FF)$. Thus, by \eqref{eq:vj=sumtj} and \eqref{eq:vj<=v*}, 
we obtain
\begin{gather*}\tau_f\leq t_1+t_2+\dots+t_{N}=\\
=\left(\frac{v(k_{0})}{z_1}+\sum_{j=1}^{N-1} 
v(k_j)\left(\frac{1}{z_{j+1}}-\frac{1}{z_{j}}\right)\right) =
\frac{v(k_{0})}{z_1}+\sum_{j=1}^{N-1} 
\frac{v(k_j)q^{1-f}}{z_{j+1}z_{j}}\leq \\
\leq \frac{1-\lambda^f}{1-\lambda}\nu^{\ast}(\mathcal{F}) 
\left(1+\sum_{j=1}^{N-1}\frac{q^{1-f}}{z_j}\right) =
\frac{1-\lambda^f}{1-\lambda}\tau^{\ast}(\mathcal{F}) 
\left(1+\sum_{k=p^{f-1}+1}^{q^{f-1}\Delta}\frac{1}{k}\right)\leq \\
\leq\frac{1-\lambda^f}{1-\lambda}\tau^{\ast}(\mathcal{F})(1+\ln \Delta-(f-1)\ln 
\lambda),
\end{gather*}
which completes the proof of Theorem~\ref{thm:ftransversal}.

\subsection{Proof of 
Corollary~\ref{cor:onexistance_of_good_fractional_transvesal}}
An $f$-fold transversal $A\subset X$ ($A$ is a multiset) easily yields a 
fractional transversal: one sets the weight $w(x)=\frac{|\{x:x\in A \}|}{f}$ 
(cardinality counted with multiplicity) for any vertex $x\in X$. The total 
weight that we get from our construction in Theorem~\ref{thm:ftransversal} is 
then 
\[
\tau^\ast(\FF)\leq
\sum_{x:x\in V} w(x)\leq
\tau^\ast (\FF)\frac{1+\ln \Delta -(f-1)\ln 
\lambda}{f(1-\lambda)}\leq \tau^\ast(\FF)(1+\varepsilon).
\]

\subsection{Running time}\label{subseq:runtime}
Let $n$ denote the number of vertices and $m$ be the number of edges of the 
hypergraph. The adjacency matrix and $f$ are the inputs of the algorithm.
As preprocessing, for each vertex, we create a list of 
edges that contain it (at most $\Delta$), which takes $nm$ operations. We keep 
track of the current state in an array $k$ of length $m$.

At each step, the following operations are performed. 
Computing the value of a vertex takes the addition of at most $\Delta$ numbers.
Thus, finding the vertex of maximal value is $n\Delta$ operations.
Picking that vertex means decreasing at most $\Delta$ entries of the array $k$ 
by one. We make at most 
$\frac{1-\lambda^f}{1-\lambda}\tau^{\ast}(\mathcal{F})(1+\ln \Delta-(f-1)\ln 
\lambda)$ steps. 

With the $\lambda=0.287643$ substitution, in total, the number of operations is 
at most
\begin{equation*}
nm+
O(\tau^{\ast}\max\{\ln \Delta, f\}\cdot \Delta n)
\leq
O(\max\{\ln \Delta, f\}\cdot \Delta nm)
.
\end{equation*}

\section{Multiple covering of space -- Proof of 
Theorem~\ref{thm:multiplecovspace}}\label{sec:geometry}

We denote by $K\sim T:=\{x\in \Red\st T+x \subseteq K\}$ the \emph{Minkowski 
difference} of two sets $K$ and $T$ in $\Red$. For $K,L\subset\Red$, and $f\geq 
1$ integer, we denote the \emph{$f$-fold translative covering number} of $L$ by 
$K$, that is, the minimum number of translates of $K$ such that each point of 
$L$ is contained in at least $f$, by $N_f(L,K)$. We denote the \emph{fractional 
covering number} of $L$ by $K$ by $N^{\ast}(L,K):=\tau^\ast(\FF)$, where 
$\FF:=\{x-K\st x\in L\}$ is a hypergraph with base set $\Red$, see details in 
\cite{N15}, or \cite{AS15}.

\begin{thm}\label{thm:multiplecovgeneral}
Let $K$, $L$ and $T$ be bounded Borel measurable sets in 
$\Red$ and let $\Lambda\subset  \Red$ be a finite set with 
$L\subseteq \Lambda+T$. Then 
\begin{equation}
\label{eq:cvxIG}
N_f(L,K)\le \left\lceil \myconst N^{\ast}(L-T,K\thicksim 
T)\max\left\{\ln\left(\max_{x\in L-K}|(x+(K\sim T))\cap \Lambda 
|\right),f\right\} \right\rceil. 
\end{equation}
If $\Lambda\subset L$, then we have 
\begin{equation}
\label{eq:cvxIGspec}
N_f(L,K)\le \left\lceil \myconst N^{\ast}(L,K\thicksim 
T)\max\left\{\ln\left(\max_{x\in L-K}|(x+(K\sim T))\cap \Lambda 
|\right),f\right\} \right\rceil.
\end{equation}

\end{thm}
Theorem~\ref{thm:multiplecovgeneral} is the $f$-fold analogue 
of \cite{N15}*{Theorem~1.2}, where the case $f=1$ is considered.
For completeness, we give an outline the proof.

\begin{proof}[Proof of Theorem~\ref{thm:multiplecovgeneral}]
To prove \eqref{eq:cvxIG}, consider the hypergraph with base set $\Red$ and 
hyperedges of the form $u-(K\sim T)$, where $u\in\Lambda$. An 
$f$-fold transversal of this hypergraph clearly yields an $f$-fold covering of 
$L$ by translates of $K$. A substitution into \eqref{eq:ftransversalcor} yields 
the desired bound. We omit the proof of \eqref{eq:cvxIGspec}, which is very 
similar.
\end{proof}

Using this result, one may prove Theorem~\ref{thm:multiplecovspace} following 
\cite{N15}*{proof of Theorem~2.1}, which is the proof of Rogers' density bound 
\eqref{eq:rogers}. We give an outline of this proof.

\begin{proof}[Proof of Theorem~\ref{thm:multiplecovspace}]
Let $C$ denote the cube $C=[-a,a]^d$, where $a>0$ is large. 
Our goal is to cover $C$ by translates of $K$ economically.
We only consider the case when $K=-K$, as treating the general case would add 
only minor technicalities.

Let $\delta>0$ be fixed (to be chosen later) and let $\Lambda\subset\Red$ be a 
finite set such that $\Lambda+\frac{\delta}{2}K$ is a saturated (ie. maximal) 
packing of $\frac{\delta}{2}K$ in $C-\frac{\delta}{2}K$. By the 
maximality of the packing, we have that $\Lambda+\delta K\supseteq C$.
By considering volume, for any $x\in \Red$ we have
\begin{equation}\label{eq:lambdasmall}
 \card{{\Lambda\cap (x+(1-\delta)K)}}\leq 
 \frac{\vol{(1-\delta)K 
+\frac{\delta}{2}K}}{\vol{\frac{\delta}{2}K}}\leq
 \left(\frac{2}{\delta}\right)^d.
\end{equation}

Let $\varepsilon>0$ be fixed. Clearly, if $a$ is sufficiently large, then
\begin{equation}\label{eq:nstarobvious}
 N^\ast(C-\delta K,(1-\delta)K)\leq
 (1+\epsilon)\frac{\vol C}{(1-\delta)^d\vol K}.
\end{equation}

By \eqref{eq:cvxIG}, \eqref{eq:lambdasmall} and \eqref{eq:nstarobvious} we have
\begin{equation*}
 N_f(C,K)\leq 
 \left\lceil
  \myconst\frac{1+\epsilon}{(1-\delta)^d}
  \frac{\vol C}{\vol K}
  \max\left\{d\ln\left(\frac{2}{\delta}\right),f\right\} 
 \right\rceil.
\end{equation*}

Thus, we obtain an $f$-fold covering of $C$. We repeat this covering 
periodically for all translates of $C$ in a tiling of $\Red$ by translates of 
$C$, which yields an $f$-fold covering of $\Red$. The density of this covering 
is at most
\begin{equation*}
 N_f(C,K)\vol{K}/\vol{C}\leq
  \left\lceil
  \myconst\frac{1+\epsilon}{(1-\delta)^d}
  \max\left\{d\ln\left(\frac{2}{\delta}\right),f\right\} 
 \right\rceil.
\end{equation*}

We choose $\delta=\frac{2}{d\ln d}$, and a standard computation yields the 
desired result.
\end{proof}

\subsection*{Acknowledgement} We thank the referees, whose comments helped 
greatly to improve the presentation.

\bibliographystyle{amsalpha}
\bibliography{biblio}
\end{document}